\RequirePackage{snapshot}
\RequirePackage{fix-cm}
\documentclass[smallextended]{svjour3}       
\smartqed  
\usepackage{graphicx}
\usepackage{amsmath}
\usepackage{amssymb}
\usepackage{mathtools}
\usepackage[shortlabels]{enumitem}
\usepackage{varioref}
\labelformat{equation}{Eqn.~(#1)}
\labelformat{figure}{Fig.~#1}
\labelformat{proposition}{Proposition~#1}
\labelformat{corollary}{Corollary~#1}
\labelformat{lemma}{Lemma~#1}
\labelformat{theorem}{Theorem~#1}
\labelformat{remark}{Remark~#1}
\labelformat{definition}{Definition~#1}\usepackage{varioref}
\labelformat{equation}{Eqn.~(#1)}
\labelformat{figure}{Fig.~#1}
\labelformat{exercise}{Exercise~#1}
\labelformat{table}{Table~#1}
\labelformat{corollary}{Corollary~#1}
\labelformat{lemma}{Lemma~#1}
\labelformat{theorem}{Theorem~#1}
\labelformat{remark}{Remark~#1}
\labelformat{definition}{Definition~#1}

\newcommand{\nfrac}[2]{\displaystyle{#1\over#2}}
\newcounter{cntr}

\newcommand{\alcghs}{\mbox{$\alpha^3/gh^2$}}

\def\pmod#1{\enspace\bigl({\rm mod}~#1\bigr)}

\newcommand{\io}[1]{{\cal O}_{ \scriptstyle#1}}
\DeclarePairedDelimiter\abs{\lvert}{\rvert}
\def\emp#1/{{\sl #1 \/}}
\def\Z{\ensuremath{\mathbb{Z}}}
\def\Q{\ensuremath{\mathbb{Q}}}
\def\N{\ensuremath{\mathbb{N}}}
\newcommand{\C}{\ensuremath{\mathbb{C}}}
\def\rat#1. #2.{\displaystyle{{#1}\over{#2}}}
\def\malign#1. #2/{\halign{&\hfill##\hfill\cr%
		#1\cr%
		#2\cr}}%
\def\bib #1. #2. #3. #4. #5\par{%
	\quad{\rm #1 \/}{\it #2, \/}%
	{\rm #3, \/}{\bf #4}{\ #5.}}%

\DeclareMathOperator{\disc}{disc}

\begin{document}

\title{Integral Basis for Quartic Kummer Extensions Over \(\mathbb{Q}[i]\)
}


\author{Dr. S. Venkataraman         \and  Prof. Manisha V. Kulkarni 
}


\institute{  S. Venkataraaman \at
	Associate Professor in Mathematics\\
	School of Sciences \\
	Indira Gandhi National Open University\\
	Maidangarhi\\
	Delhi 110068
	\email{svenkat@ignou.ac.in
	}
           \and
           Prof. Manisha V. Kulkarni \at
           Professor in Mathematics\\
           International Institute of Information Technology\\
           26/C, Phase-1, Electronics City\\
           Hosur Road\\
           560100\\
           \email{manisha.kulkarni@iiitb.ac.in
         }  
}

\date{Received: date / Accepted: date}

\maketitle

\begin{abstract}
Let \(K=\Q[\iota]\) and \(N=K[\sqrt[4]{\alpha}]\), \(\alpha\in\Z[\iota]\), \(\alpha=fg^2h^3\), \(f\), \(g\), \(h\in\Z[\iota]\) are pairwise coprime and square free. Let \(\io{N}\) be the ring of integers of \(N\). 
In this article we construct normalised integral basis for \(\io{N}\) over \(\Z[\iota]\), that is an integral basis of the form 
\[
\left\{1,\frac{f_1(\alpha)}{d_1},\frac{f_2(\alpha)}{d_2},\frac{f_{3}(\alpha)}{d_3}\right\}
\]
where \(d_i\in\Z[i]\) and \(f_i(X)\), \(1\leq i\leq 3\) are monic polynomials of degree \(i\) over \(\Z[\iota]\).  We explicitly determine what \(d_i\), \(1\leq i\leq n-1\) are in terms of \(f\), \(g\) and  \(h\).
\end{abstract}

\section{Introduction}\label{intro}
Let \(L/F\) be an extension of number fields and let \(\io{F}\) and \(\io{L}\) be the ring of integers of \(L\) and \(F\), respectively.  Then, \(\io{L}\) has a natural structure of a module over \(\io{F}\). When \(F=\Q\), since \(\io{\Q}=\Z\) is a PID, \(\io{L}\) is free over \(\Z\). So, \(\io{L}\) has a  basis over \(\Z\) that we call the integral basis for \(\io{L}\).  Dedekind gave an integral basis for cubic fields in \cite{dedekind1900}.  Berwick computed integral bases for many extensions of \(\Q\) in \cite{berwick1927}. There has been a lot of work on topic and the literature is too vast to give an exhaustive list.  

When \(F\neq \Q\)  the module  \(\io{L}\) may not be free over \(\io{F}\).  H. B. Mann has discussed the question of existence of integral basis in \cite{hbmann}. In this paper, we discuss the case where \(F=\Q[\iota]\) and \(L=\Q\left(\sqrt[4]{\alpha}\right)\) where \(\alpha\in \Q[i]\) and is not a fourth power.  Hymo and Parry, in \cite{hp92} have already calculated the integral basis when $\alpha\in\Q$. Funakura \cite{funakura} has given integral basis when 
\(F=\Q\) and \(L=\Q\left(\sqrt[4]{k}\right)\), \(k\in \Z\). In \cite{jhakar2021}, the case \(F=\Q\), \(L=F\left(\sqrt[n]{a}\right)\) is dealt with, where \(a\in \Z\), \(n\in \N\) is such that, for each prime \(p\) dividing \(n\), either \(p\nmid a\) or  the exponent of \(p\) in \(a\) is coprime to \(p\).

In Section~\ref{section-2}, we recall some of the results that we need in proving the main result.  In Section~\ref{section-3}, we prove the main result of our paper.

\section{Preliminaries}\label{section-2}
Throughout this paper, let \(f\), \(g\), \(h\) be pairwise coprime, square free integers in \(\Z[\iota]\), \(m=fg^2h^3\) and \(\alpha=\sqrt[4]{fg^2h^3}\). Let \(K=Q[\iota]\), \(M=\Q\left(\sqrt{fh}\right)\) and \(N=K[\alpha]\). In this paper, we determine a normalised integral basis for \(N/K\). (See definition below.)

Let $\io{F}$ be a P.I.D with quotient field $F$, where $F$ is a finite
extension of $\Q$. Let $L$ be an extension of degree $n$ over $F$ and $\io{L}$, the
integral closure of $\io{F}$ in $L$. Let $\alpha\in \io{L}$ be such that
$F(\alpha)=L$. We have the following result.
\begin{theorem}[Normalised Integral Basis]\label{main-theorem}
	There exist $d_{1},d_{2},\ldots,d_{n}\in\io{F}$ and monic polynomials
	$f_{i}(X)\in \io{F}[X]$, $1\leq i\leq n-1$, $deg(f_{i}(X))=i$, such that 
	\begin{equation}\left\{ 1,\frac{f_{1}(\alpha)}{d_{1}},
		\frac{f_{2}(\alpha)}{d_{2}}\cdots,
		\frac{f_{n-1}(\alpha)}{d_{n-1}} \right\}\label{NIM}\end{equation} is a basis for $\io{L}$ over $\io{F}$. Further,
	$d_{i}$'s satisfy the following conditions:
	\begin{enumerate}[(1)]
		\item Each $f_{i}$ can be replaced by any monic polynomial $g\in\io{K}[X]$ of the same
		degree such that $\displaystyle{g(\alpha)\over{d_{i}}}
		\in S$.\label{basis3}
		\item If $\displaystyle{g(\alpha)\over q}\in \io{L}$ for some monic polynomial
		$g(X)\in \io{K}[X]$ of degree $i$ and $q\in \io{K}$, $q\mid d_{i}$.\label{basis5}
		\item $d_{i}d_{j}\mid d_{i+j}$ if $i+j< n$.\label{basis1}
		\item $(d_{1}d_{2}\ldots d_{n-1})^{2}\disc(L/F)=\disc(\alpha)$.\label{basis2}
		\item $d_{1}^{n(n-1)}\mid \disc(\alpha)$.\label{basis4}
	\end{enumerate}
\end{theorem}
	A. A. Albert calls the integral basis of the form given in \ref{NIM} a \textit{Normalised Integral Basis} in the case \(K=\Q\). See \cite{albert37}.  He attributes this result to \cite{berwick1927}. This result seems fairly well known for \(K=\Q\)  and can be easily extended to any number field whose ring of integers is a PID. (For \(K=\Q\), see \cite{mar}, page 26 and exercises 38, 39 and 40 in pages 34 and 35. See also    \cite{alaca_williams_2003}, page 169 though the formulation is slightly different.)  We give a complete proof of this for the case where \(K\) is a PID in the Appendix. 

We now state a result proved in \cite{mar}, page 21, proved in the case \(K=\Q\).
\begin{lemma}\label{key-lemma}
	Let \(L/F\) be an extension of nunmber fields, \([L:K]=n\). Further, suppose that \(\io{K}\) is a PID.  Let 
	\[
	\left\{\alpha_1,\alpha_{2},\cdots,\alpha_n\right\}\subset S
	\]
	be a basis for \(L/F\) and 
	\[
	d=\disc\left(\alpha_1,\alpha_2,\cdots,\alpha_n\right)=\left\vert
	Tr^{L}_K\left(\alpha_i\alpha_j\right)\right\vert^2
	\]
	Then, every \(\alpha\in\io{L}\) can be expressed in the form 
	\begin{equation}\label{eq:key-lemma}
		\frac{m_1\alpha_1+m_2\alpha_2+\cdots+m_n\alpha_n}{d} 
	\end{equation}
with all \(m_j\in\io{K}\) and all \(m_j^2\) divisible by \(d\).
\end{lemma} 

See \cite{mar}, Chapter 2, Theorem 9 for the case \(K=\Q\). The same proof carries over here.

\subsection{Quadratic Sub Extension}\label{subsec-2-1}
In this subsection, we determine integral basis and discriminant for \(M\) over \(K=\Q[\iota]\). We let \(\gamma=\sqrt{fh}\) where \(f\) and \(h\) are square free and coprime to each other. Let \(\Delta\) be the discriminant of \(M/K\).

Given \(\alpha\in\io{K}\), we can write \(\alpha=\alpha_1\alpha_2\) such that \(\alpha_1\) is coprime to \(1+\iota\) and  the only prime divisor of \(\alpha_2\) is \(1+\iota\). We call \(\alpha_1\) the \textbf{odd part} of \(\alpha\) and \(\alpha_2\) the \textbf{even part} of \(\alpha\). Of course \(\alpha_1\) and \(\alpha_2\) are determined only up to a unit.  We will call a prime \(\pi\in\io{K}\) an odd prime if it is coprime to \(1+\iota\). We will use the \vref{prop_1} to determine an integral basis for \(M/K\).

\begin{proposition} \label{prop_1}The extension \(M/K\) has an integral basis of the form \(\{1,\frac{a+\gamma}{d_1}\}\) with the following properties:
	\begin{enumerate}[(1)]
		\item \(a^2\equiv fh \pmod{d_1^2}\) where \(d_1^2\mid 4\).\label{quad1}
		\item \(\Delta=\frac{4fh}{\strut d_1^2}\).\label{quad2}
		\item If \(a'\equiv a\pmod{d_1}\), then \(\{1,\frac{a'+\gamma}{d_1}\}\) is also an integral basis for \(M/K\).\label{quad3}
	\end{enumerate}
	\end{proposition}
\begin{proof} 
	
	By \ref{main-theorem}  there is an integral basis of the form $\left\{1, \nfrac{a+\gamma}{d_1}\right\}$, \(a\in\Z[\iota]\) for \(M/K\). Since \(\frac{a+\gamma}{d_1}\in \io{M}\) it follows that \(N_{M/K}\left(\frac{a+\gamma}{d_1}\right)=\frac{a^2-fh}{d_1^2}\in \io{K}\), therefore \(a^2\equiv fh\pmod{d_1^2}\). 

        Further \(\disc_{M/K}(\gamma)=4fh\).  Then, from property \ref{basis2} of \ref{main-theorem} it follows that \(d_1^2\Delta=4fh\). So, property \ref{quad2} follows.
        
        We now show that \(d_1^2\mid 4\).  From \ref{main-theorem}, \vref{basis4}, it follows that \(d_1^2\mid 4fh\). If \(\pi\) is an odd prime factor of \(d_1\), \(\pi\nmid4\), and  \(\pi^2\mid fh\) which is not possible since \(f\) and \(h\) are square free and pairwise coprime. So, \(d_1^2\) is a power of \(1+\iota\). This completes the proof of \ref{quad1}. 
        
        Since \(f\) and \(h\) are square free and pairwise coprime, the maximum power of \(1+\iota\) dividing \(4fh\) is five.  Since \(d_1^2\) is an even power of \(1+\iota\), \(d_1^2\mid 4\).

If \(a'\equiv a\pmod{d_1}\), we have \(\frac{a'+\gamma}{d_1}=\frac{a'-a}{d_1}+
\frac{a+\gamma}{d_1}\in\io{M}\)  since \(d_1\mid\left(a-a'\right)\). Therefore, by \ref{basis3}, \ref{main-theorem}, it follows that \(\{1,\frac{a'+\gamma}{d_1}\}\) is also an integral basis.
\end{proof}
From \ref{prop_1}, \ref{quad1}, it follows that \(d_1\) is \(1\), \(1+\iota\) or \(2\). Note further that, \(d_1=2\) iff \(fh\) is a square modulo four, \(d_1=(1+i)\) iff \(fh\) is a square modulo 2, but not a square modulo 4.  We have \(d_1=1\) iff \(fh\) is not a square modulo 2.

If \(d_1=2\), \(fh\) has to be a square modulo 4. 
The only squares \(\pmod{4\Z[\iota]}\) are \(\pm 1\). If \(fh\equiv 1\pmod{4\Z[\iota]}\), \(\left\{1,\nfrac{1+\sqrt{fh}}{2}\right\}\) is an integral basis and the discriminant is \(fh\).  If \(fh\equiv -1\pmod{4\Z[\iota]}\), \(\left\{1,\nfrac{i+\sqrt{fh}}{2}\right\}\) is an integral basis and the discriminant is again \(fh\).

If \(d_1=1+i\), \(fh\) has to be a square modulo 2 and not a square modulo 4.  Note that \(fh\not\equiv 0\pmod{2\Z[\iota]}\) since \(fh\) is square free and \(f\) and \(h\) are pairwise coprime.

The squares \(\pmod{2\Z[\iota]}\), but non squares \(\pmod{4\Z[\iota]}\) are \(1+2\iota\), \(-1+2\iota\).  If \(fh\equiv 1+2\iota\pmod{4\Z[\iota]}\), an integral basis is \(\left\{1,\nfrac{1+\sqrt{fh}}{1+\iota}\right\}\), and the discriminant is \(2fh\). If \(fh\equiv 3+2\iota\pmod{4\Z[\iota]}\), \(\left\{1,\nfrac{-1+\sqrt{fh}}{1+\iota}\right\}\) is an integral basis and the discriminant is \(2fh\).  

If \(d_1=1\), then \(fh\) is not a square modulo 2. We have that \(\iota\), \(3\iota\), \(1+\iota\), \(1+3\iota\), \(2+\iota\), \(2+3\iota\) and \(3+3\iota\) not squares \(\pmod{2\Z[\iota]}\). Either they are \(0\pmod{(1+\iota)}\) or \(i\pmod{2\Z[\iota]}\).  In both the cases, the integral basis is \(\left\{1,\sqrt{fh}\right\}\) and the discriminant is \(4fh\).  We summarise the result in \vref{quadratic-basis}. 
\begin{table}[htb]
	\centering
	$\begin{array}{lcc}
		\hline
		\textrm{Case}&\textrm{Integral basis}&\textrm{Discriminant}\\
		\hline
		&&\\[-7pt]
		fh\equiv 1\pmod{4}& \left\{1, \nfrac{1+\sqrt{fh}}{2}\right\}& fh\\
		fh\equiv -1\pmod{4}&\left\{1, \nfrac{\iota+\sqrt{fh}}{2}\right\}& fh\\
		fh\equiv 1+2\iota\pmod{4}&\left\{1, \nfrac{1+\sqrt{fh}}{1+\iota}\right\}& 2fh\\
		fh\equiv -1+2\iota\pmod{4}&\left\{1, \nfrac{\iota+\sqrt{fh}}{1+\iota}\right\}&
		2fh\cr 
		fh\equiv\iota\pmod{2} \textrm{ or } 0\pmod{1+\iota}& \left\{ 1,
		\sqrt{fh}\right\}& 4fh\\
		\hline
	\end{array}$
\caption{\label{quadratic-basis} Basis and discriminant in the quadratic case}
\end{table}
We recall some terms and facts from algebraic number theory.
\begin{definition}\label{def-ramification} Let \(L/F\) be an extension of number fields of  degree \(n\) and let \(P\) be a prime ideal in \(\io{K}\). Let 
	\[
	P=Q_1^{e_1} Q_2^{e_2}\ldots Q_r^{e_r}
	\] be the decompostion of \(P\) into prime ideals in \(\io{K}\). We say that \(P\) is \textbf{ramified} in  \(L/F\) if \(e_i>1\) for some \(i\), \(1\leq i\leq r\). We say that the prime \(P\) is \textbf{tamely ramified} in \(L/F\) if the characteristic of  \(\left(\nfrac{\io{K}}{P}\right)\) does not divide \(e_i\) for any \(i\), \(1\leq i\leq r\).
\end{definition}
In general, we have 
\begin{equation}\label{ramification-relation}
	\sum_{i=1}^re_i(Q_i\vert P)f(Q_i\vert P)=n
	\end{equation}
where \(Q_i\) and \(P\) are as in \vref{def-ramification}. 
Recall that, if \(L/F\) is a galois extension, all the \(e(Q_i\vert P)\) and \(f(Q_i\vert P)\) are equal. Suppose \(L/F\) is a galois extension and \(e(Q_i\vert P)=e\), \(f(Q_i\vert P)=f\) for\(1\leq i\leq r\). Then,\ref{ramification-relation} becomes 
\begin{equation}\label{ramification-relation-a}
	ref=n
\end{equation}
The next result gives equivalent conditions for a prime to be tamely ramified in a galois extension of number fields.
\begin{proposition}\label{tame-ramification} Let \(L/F\) be a galois extension of algebraic number fields and \(P\) be a prime in \(\io{K}\) with ramification index \(e\).  Then, following statements are equivalent:
	\begin{enumerate}[(1)]
		\item \(P\) is tamely ramifed in \(L/F\).
		\item \(P\) doesn't divide \(Tr_{L/F}(\io{L})\).
		\item If \(D\) is the different of \(L/F\), \(\nu_Q(D)=e-1\) for any prime \(Q\) in \(\io{L}\) with \(Q\cap \io{L}=P\).
	\end{enumerate}
\end{proposition}
See \cite{cassels-frohlich-old}, Theorem 2, page 21 for the local case. We can easily deduce \ref{tame-ramification} from this. 

We now recall  a result regarding discriminant.
\begin{proposition}
	Let \(F\subset F'\subset L\).  Then
	\begin{equation}
		\disc(L/F)=N_{F'/F}\left(\disc(L/F')\right)(\disc(F'/F))^{[L:F']}
	\end{equation}
	\end{proposition}
\subsection{Ramification Groups}
In this subsection we gather various results on ramification groups that we will use in the proof the  main result. We fix a prime \(p\). Let \(F\subset L\)   be finite extensions of \(Q_p\), the completion of \(Q\) with respect to the discrete valuation defined by the prime \(p\).
Let \(\nu_L\) and \(\nu_F\) be the normalised valuation on \(L\) and \(F\), respectively, such that \(\nu_L(L^*)=\Z\) and \(\nu_F(F^*)=\Z\). We let \(\io{K}=\left\{x\in K^*\vert \nu_K(x)\geq 0\right\}\). We define \(\io{L}\) similarly. Then, there is a \(\theta\in\io{L}\) such that \(\{1,\theta,\ldots,\theta^{n-1}\}\) is a basis for \(\io{L}\) over \(\io{K}\). If \(L/K\) is totally ramified, \(\{1,\pi,\ldots,\pi^{n-1}\}\) is a basis for \(\io{L}\) over \(\io{K}\), where \(\pi\) is a generator of the unique prime ideal of \(\io{L}\). We also call \(\pi\) a uniformiser in \(L\).

Let \(G\) be the galois group of \(L/F\). We  define \(i_G\colon G\to \Z\)  by \(i_G(\sigma)=\nu_L(\sigma(\theta)-\theta)\)  
\begin{definition}\label{ramification-groups}
 Then, for \(i\geq 0\), we define
	\begin{equation}
		G_i=\left\{\sigma\in G\mid i_G(x)\geq i+1 \right\}\label{def:ramification-groups}
	\end{equation} 
	
		\end{definition}
	
Recall that \(e(Q\vert P)=\vert G_1\vert\). Further \(G_{i+1}\subset G_i\) and \(G_i=\{e\}\) for sufficiently large \(i\). 

We  have the following formula for the power of \(Q\) dividing the different of \(L/F\) in terms of the ramification groups.
\begin{proposition} Let \(L/F\) be a galois extension. Let \(D\) be the different of \(L/K\).  Then,
	\begin{equation}
		\nu_L(D)=\sum_{\sigma\neq 1}i_G(\sigma)=\sum_{i=0}^\infty\left(\vert G_i\vert-1\right)\label{different-formula}
	\end{equation}
	\end{proposition}
See \cite{lf}, page 64. 

	Suppose \(K\subset K'\subset L\) and \(K'/K\) is a normal extension so that \(H=G(L/K')\) is a normal subgroup of \(K'\). We have the following result:
\begin{proposition}
	Let \(G\) be the galois group of \(L/F\) and let \(H=G_j\) for some \(j\geq 0\).  Then \((G/H)_i=G_i/H\) for \(i\leq j\) and \((G/H)_i=\{1\}\) for \(i\geq j\). 
\end{proposition}
See Corollary to Proposition 3, page 64 in \cite{lf}.

If  an integer \(i\) is such that \(G_i/G_{i+1}\neq \{1\}\), then \(i\) is called a \textbf{(lower) break number.}  

We also recall the definition of the \textbf{Herbrand function.}  For real \(t\) let \(G_t=G_{\lceil t\rceil}\) where \(\lceil x\rceil\) is the largest integer \(\geq x\).   We then define
\begin{equation}\phi(u)=\int_0^u\frac{1}{[G_0:G_t]}\,dt\textrm{ for }u\geq 0\label{herbrand-function}\end{equation}
Then, \(\phi\) is a continuos, piecewise linear and an increasing function and therefore has an inverse \(\psi(u)\).

We define the \textbf{upper numbering} of the ramification groups by \(G^t=G_{\psi(t)}\) or \(G_t=G^{\phi(t)}\).  A real number \(t\) is an \textbf{upper break number} if \(G^t>G^{t+\epsilon}\) for \(\epsilon>0\). If \(L/F\) is an abelian extensions, the upper break numbers are integers.  We have \(G^{b^i}=G_{b_i}\) for each \(i\geq -1\). Further, the upper numbering behaves well with respect to passing to the quotient. If \(H\) is a normal subgroup of 
\(G\), we have \begin{equation}
	(G/H)^v=G^vH/H\label{quotient}
\end{equation}
We now recall some known  results for convenience.
Let \(e=\nu_L(p)\) be the absolute ramification index of \(L\). 
\begin{proposition} Let Suppose \(i>e/(p-1)\).  Then, \(G_i=\{1\}\)\label{serre-exercise}
\end{proposition}
See \cite{lf}, exercise 3 c), page 72.
We have the following result:
\begin{proposition}\label{prop:7a}
	Let \(G\) be the galois group of \(L/F\) and let \(H=G_j\) for some \(j\geq 0\).  Then \((G/H)_i=G_i/H\) for \(i\leq j\) and \((G/H)_i=\{1\}\) for \(i\geq j\). 
\end{proposition}
See Corollary to Proposition 3, page 64 in \cite{lf}. Furtherm we have the following result:
\begin{proposition}\label{prop:8a} The integers \(i\), \(i\geq 1\), such that \(G_i\neq G_{i+1}\) are congruent to each other \(\bmod{p}\) where \(p\) is the charcteristic of the residue field \(L\).
	\end{proposition}
See \cite{lf}, Proposition 11, page 70.

We get more detailed information when we assume that \(L/F\) is totally ramified cyclic extension of prime power degree.
\begin{proposition}\label{Wyman}
	Let Suppose \(L/F\) ia totally ramified, cyclic extension of degree \(p^n\), \(n\geq 1\). Let \(e'\) be the absolute ramfication index of \(F\) and \(b^1>e'/(p-1)\).  Then, the upper break numbers are precisely \(b^i=b^1+(i-1)e'\).\label{wyman}
\end{proposition}
See \cite{wyman}, Theorem 3.

We continue to assume that \(L/F\) is a totally ramified, cyclic extension of order \(p^n\), \(n\in \N\).   Let \(G(i)\) be the unique subgroup of of \(G\) of order \(p^{n-i}\).  We know that each \(G(i)\) is a ramification group of \(L/F\).    Then, there are strictly positive integers \(i_0\), \(i_1\), \(\ldots\), \(i_{n-1}\) such that 
\[
	\begin{gathered}
		G_0=\cdots=G_{i_0}=G=G^0=\cdots=G^{i_0}\\
		G_{i_0+1}=\cdots=G_{i_0+pi_1}=G(1)=G^{i_0+1}=\cdots=G^{i_0+i_1}\\
		G_{i_0+pi_1+1}=\cdots=G_{i_0+pi_1+p^2i_2}=G(2)=G^{i_0+i_1+1}=\cdots=G^{i_0+i_1+i_2}\\
		\vdots\\
		G_{i_0+pi_1+\cdots+p^{n-1}i_{n-1}+1}=\{1\}=G^{i_0+i+1+\cdots+i_{n-1}+1}
	\end{gathered}
\]
Therefore, if \(i_0\), \(i_1\), \(\ldots\), \(i_{n-1}\) are given by the above equations we have  
\begin{align*}
	b_1&=i_0 &b^1&=i_0\\
	b_2&=i_0+pi_1& b^2&=i_0+i_1\\
	&\vdots&&\vdots\\
	b_{n}&=i_0+pi_1+\cdots p^{n-1}i_{n-1} &b^{n-1}&=i_0+i_1+\cdots+i_{n-1}
\end{align*}
	It follows that, given the upper numbering for \(L/K\), we can calculate the lower numbering recursively as follows:
\begin{align}
	b_1&=b^1\label{eqn:10-a}\\
	b_2&=b_1+p\left(b^2-b^1\right)\label{eqn:10-b}\\
	&\vdots\nonumber\\
	b_{n}&=b_{n-1}+p^{n-1}\left(b^{n}-b^{n-1}\right)\label{eqn:10-c}
\end{align}
\section{Integral Basis}\label{section-3}
In this section we determine an integral basis for the extension $N/K$.

Since $\Z[\iota]$ is a P.I.D, there is a basis \begin{equation}\label{normalised-integral-basis}\left\{1,
\alpha_{2}={a_0+\alpha\over{d_{1}}},\alpha_{3}={\displaystyle 
	b_0+b_1\alpha+\alpha^{2}\over
	{\displaystyle d_{2}}},\alpha_{4}={\displaystyle c_0+ 
	c_1\alpha+c_2\alpha^{2} +\alpha^{3}\over{\displaystyle d_{3}}}\right\}\end{equation} with \(a_0)\), \(b_0\), \(b_1\), \(c_0\) \(c_1\), \(c_2\) and \(c_3\)  in \(\io{K}\). We will determine
$d_{1},d_{2},d_{3}$. 

\begin{remark} \label{remark-1}From \ref{main-theorem}, \vref{basis5}, it  follows that  \(\alpha_2\) can be  replaced by \(\nfrac{a_0'+\alpha}{d_1}\) for any \(a_0'\) with \(a_0'\equiv a_0\pmod{d_1}\). Similarly, \(\alpha_3\) can be replaced by \(\nfrac{b_0'+b_1'\alpha^2+\alpha^3}{d_2}\) with \(b_0\equiv b_0'\pmod{d_2}\), \(b_1'\equiv b_1\pmod{d_2}\) and \(\alpha_4\) can be replaced by
  \(\nfrac{c_0'+c_1'\alpha+c_2'\alpha^2+
  \alpha^3}{d_3}\) with \(c_0'\equiv c_0\pmod{d_3}\), \(c_1'\equiv c_1\pmod{d_3}\), \(c_2'\equiv c_2\pmod{d_3}\).
  \end{remark}

We have \[\disc (\alpha)=N_{N/K}(4\alpha^3)=4^4\alpha^3(-\iota\alpha^3)(-\alpha^3)(-\iota\alpha^3)=(1+\iota)^{16}f^3g^6h^9\]
From the property \ref{basis2} of \ref{main-theorem}, it follows that 
\begin{equation}\label{main-equation}
	\left(d_1d_2d_3\right)^2\disc(N/K)=(1+\iota)^{16}f^3g^6h^9
\end{equation}

 With this terminology, let \(\delta_1\), \(\delta_2\) and \(\delta_3\) be the even parts  of \(d_1\), \(d_2\) and \(d_3\), respectively. Let \(\mu_1\), \(\mu_2\) and \(\mu_3\) be the odd parts of \(d_1\), \(d_2\) and \(d_3\), respectively. 

The next lemma is a purely computational result that we will use later.
\begin{lemma}
	Let \(\beta_3=\nfrac{b_0+b_1\alpha+\alpha^2}{\delta_2}\) and \(\beta_4=\nfrac{c_0+c_1\alpha+c_2\alpha^2+c_3}{\delta_3}\) where \(b_0\), \(b_1\), \(c_0\), \(c_1\) and \(c_2\in\io{K}\). Then, we have \(\beta_3\), \(\beta_4\in \io{N}\) and 
	\begin{align}
		Tr_{N/M}\left(\beta_3\right)&=2\left(\nfrac{b_0+\alpha^2}{\delta_2}\right)\label{beta-3-tr}\in\io{M}\\ N_{N/M}\left(\beta_3\right)&=\nfrac{\left(b_0^2+m\right)+\left(2b_0-b_1^2\right)\alpha^2}{\delta_2^2}\label{beta_3-nr}\in\io{M}\\
		Tr_{N/M}\left(\beta_4\right)&=2\left(\nfrac{c_0+c_2\alpha^2}{\delta_3}\right)\in\io{M}\label{beta-4-tr}\\
		N_{N/M}\left(\beta_4\right)&=\nfrac{\left(c_0^2+c_2^2m-2c_1m\right)+\left(2c_0c_2-c_1^2-m\right)}{\delta_3^2}\in\io{M}\label{beta-4-nr}
	\end{align}
	\end{lemma}
	\begin{proof}
		We have \(\beta_3=\mu_2\alpha_3\) and \(\beta_4=\mu_3\alpha_4\). Since \(\alpha_2\), \(\alpha_3\), \(\mu_2\) and \(\mu_3\) are algebraic integers, so are \(\beta_3\) and \(\beta_4\).  The remaining part of the proof is a simple computation.
		\end{proof}
Before we find the integral basis, we gather some facts that we will use in the proof as a proposition.
\begin{proposition}\label{prop-3}
		\begin{enumerate}[(1)]
			\item If \(\pi\neq 1+\iota\) is any prime, then \(\pi\) doesn't divide \(d_1\). Further, \((1+\iota)^2\nmid d_1\) if \(1+\iota \nmid h\). \label{p-2}
				\item If \(\pi\) is a prime in \(\io{K}\) and \(\pi\) divides \(f\), \(g\) or \(h\), then \(\pi\) is ramified in \(N/K\).\label{I} 
			\item Let  \(\pi\neq 1+\iota\). If \(\pi\mid fh\), 
			\(\nu_\pi(\disc(N/K)))=3\). If \(\pi\mid g\) 
			\(\nu_\pi(\disc(N/K))=2\).\label{p-1}
                      \item We have \(gh\mid d_2\) and \(gh^2\mid d_3\).  Further, the odd part of \(gh\) is the same as the odd part of \(d_2\)  and the odd part of \(gh^2\) is the same as the odd part of \(d_3\).\label{p-0}
                        	\item $\delta_{1}\vert 1+\iota$ and $\delta_{1}=1+\iota$ if and only if
		$m\equiv 1\pmod{4\io{K}}$.\label{III}
			\item If $m$ is odd \(\delta_3\mid 4\). Also,  $N/K$ is unramified at \((1+\iota)\) iff \(\delta_3=4\). \label{II}
	\end{enumerate}
	\end{proposition}
\begin{proof}\begin{enumerate}[(1)]
  \item If \(\pi\) divides \(d_1\), \(\pi^{12}\) divides \(\disc(\alpha)\) from property \(\ref{basis5}\) of \ref{main-theorem}. Also, \(\pi\) can divide only one of \(f\), \(g\) or \(h\)  since \(f\), \(g\) and \(h\) are pairwise coprime. Since \(f\), \(g\) and \(h\) are square free \(\pi^{12}\) can't divide \(f^3\), \(g^6\) or \(h^9\).

    If \((1+\iota)^2\mid d_1\), \((1+\iota)^{24}\) divides the RHS of \ref{main-equation}. If \(1+\iota\nmid h\), the maximum power of \(1+\iota\) that divides the RHS of \ref{main-equation} is 22. (This will happen when \(1+\iota\mid g\).) 
\item If \(\pi\mid fh\), from the computation of the disriminants in \ref{quadratic-basis} \(\pi\) is ramified in \(M/K\) and hence ramified in \(N/K\).  Let \(\pi\) divide \(g\). We have
 \begin{equation}\label{eq-4}\left(\alpha^4\right)=\left(fg^2h^3\right)\subset \left(\pi\right)\end{equation}
 in \(N\).  If \(\pi\) is unramified in \(N/K\), \(f=1\) and \(r=4\) then 
 \(\left(\pi\right)=Q_1Q_2Q_3Q_4\). If \(r=2\), \(f=2\) we have 
 \(\left(\pi\right)=Q_1Q_2\)   and \((\pi)\) is inert if \(r=1\), \(f=4\).
 
 In the first case, from \ref{eq-4} it follows that 
 \((\alpha^4)=Q_1Q_2Q_3Q_4J\) for some ideal \(J\) in \(\io{N}\).  Hence,  we have   \(\alpha^4\in Q_i\), \(1\leq i\leq 4\). It follows that \(\alpha\in Q_i\) for \(1\leq i\leq 4\). Therefore, \((\alpha)=Q_1Q_2Q_3Q_4J'\) for some ideal \(J'\) in \(\io{N}\). So, \(N_{N/K}\left(Q_1Q_2Q_3Q_4\right)\mid \left(N_{N/K}((\alpha))\right)\), i.e. \(\pi^4\mid N_{M/K}(\alpha)\). Therefore, \(\pi^{12}\mid N_{M/K}\left(\alpha^3\right) \). If \(\pi\) is odd, this contradicts the fact that \(g^6\) divides \(\disc(\alpha)\) and no higher power of \(g\) divides \(\disc(\alpha)\).  
 
 If \(\pi=(1+i)\), the same argument gives that  \(\pi^{28}\mid N_{N/K}(4\alpha^3)\) since \(\pi^{16}\mid N_{N/K}(4)(=4^{4})\) and \(\pi^{12}\mid N_{N/K}(\alpha^3)\).   This is again a contradiction becasue the power of \(\) dividing \(\disc(\alpha)\) is \(22\), 16 from \((1+\iota)^{16}\) and 6 from \(g^6\).  The proofs  for the cases \(r=2\), \(f=2\) and \(r=1\), \(f=4\) are  similar.
\item  From \ref{I}, we know that \(\pi\) is ramified in \(N/K\).  So, the ramification group of \(N/K\) is nontrivial. Suppose \(\pi\mid f\) or \(h\). If the ramification group of \(\pi\) has order \(2\), the fixed field \(M'\) of the ramification group is a subfield of \(N/K\) in which \(\pi\) is unramified.  However, the galois group of \(N/K\) is cylic and \(M/K\) is the unique quadratic sub-extension of \(N/K\) and \(\pi\) is ramified in \(M/K\).  So, if follows that the ramification index of \(\pi\) in \(N/K\) is four. If \(D\) is the different of \(N/K\) and \(Q\) in \(N\) is such that \(Q^4=(\pi)\), it follows from \ref{tame-ramification} that \(\nu_{Q}(D)=3\).  So, 
\(\nu_\pi(\disc(M/K))=v_\pi\left(N_{M/K}(D)\right)=\nu_Q(D)=3\).  

If \(\pi\mid g\), then \(\pi\) is unramified in \(M/K\) since we know from \ref{quadratic-basis} that \(\pi\)  doesn't divide \(\disc(M/K)\). So, \(\pi\) is ramified only in \(N/M\) and \(e=2\) in this case.  So, if \(Q\) is such that \(Q\cap \io{K}=(\pi)\), \(\nu_Q(D)=1\).  If \(f\left(Q\mid (\pi)\right)=2\), from the relation \(f\nu_{Q}(D)=\nu_\pi(N_{N/K}(D))=\nu_\pi(\disc(M/K))\) it follows that \(\nu_{\pi}(\disc{N/K})=2\).  If \(f=1\), there will be two primes \(Q_1\) and \(Q_2\) in \(N\) such that \(Q_1\cap \io{K}=\pi\), \(Q_2\cap \io{K}=\pi\). So, again, \(\nu_\pi(\disc(N/K))=2\).
\item 	By looking at their norm and trace over \(M\), we see that \(\nfrac{\alpha^2}{gh}\) and \(\nfrac{\alpha^3}{gh^2}\) are in \(\io{N}\). It follows from property \ref{basis2} of the \ref{main-theorem} that \(gh\mid d_2\) and \(gh^2\mid d_3\).  
  Let \(\pi\neq 1+\iota\), \(\pi\mid h\).  From property \ref{p-1} of \ref{prop-3}, it follows that \(\nu_\pi(\disc(N/K))=3\). Comparing the LHS and the RHS of  \ref{main-equation}, we get
  \(
    2\nu_\pi\left(d_2d_3\right)+3=9
    \) or \(\nu_\pi\left(d_2\right)+\nu_\pi\left(d_3\right)=3\). Since \(h\mid d_2\) and \(h^2\mid d_3\),  we have \(\nu_\pi\left(d_2\right)=1\), \(\nu_\pi\left(d_3\right)=2\). So,  \(\pi^2\nmid d_2\), \(\pi^3\nmid d_3\). Similarly, we can show that, if  \(\pi\mid g\), from the fact that \(\nu_\pi(\disc(\alpha))=6\) and \(\nu_\pi(\disc(N/K))=2\),  \(\pi^2\nmid d_2\) and \(\pi^3\nmid d_3\).
\item We have $\delta_{1}^{12}\vert \disc(\alpha)$. If \(\delta_1=2\), from \ref{main-equation} it follows that \((1+\iota)^{24}\mid disc(\alpha)\). However the highest power of \(1+\iota\) dividing \(disc(\alpha)\) is \(21\), 12 from \((1+\iota)^{12}\) and  nine from \(g^9\).  

If \(\alpha_2\) is an integer,   we have $N_{N/K}(\alpha_{2})=\nfrac{(a_0^4-m)}{4}\in \io{K}$. Since the only odd fourth power $\pmod{4\io{K}}$ is 1, the result  follows.

Conversely suppose suppose \(m\equiv 1\pmod{4\io{K}}\). Consider \(\alpha_2=\nfrac{1+\alpha}{1+i}\in N\).  To prove that \(\alpha_2\in\io{N}\), we need to prove that \(N_{N/M}\left(\alpha_2\right)\) and \(Tr_{N/M}\left(\alpha_2\right)\) are integers.  We have \(Tr_{N/M}\left(\alpha_2\right)=\frac{2}{1+\iota}=1+\iota\in\io{M}\).  Also, \(N_{N/M}\left(\alpha_2\right)=\nfrac{1-\alpha^2}{2}\).  We have \(N_{M/K}\left(\nfrac{1-\alpha^2}{2}\right)=\frac{1-m}{4}\in \io{K}\) and \(Tr_{M/K}\left(\nfrac{1-\alpha^2}{2}\right)=1\in \io{K}\).  So, it follows that \(\nfrac{1-\alpha^2}{2}\in \io{M}\).

  \item Suppose \(m\) is odd. We have 
\[\alpha\beta_3=\alpha\left(\nfrac{c_0+c_1\alpha+c_2\alpha^2+\alpha^3}{\delta_3}\right)=\nfrac{c_0\alpha+c_1\alpha^2+c_2\alpha^3+m}{\delta_3}\in \io{N}\] Therefore, \begin{equation}Tr_{N/K}\left(\nfrac{c_0\alpha+c_1\alpha^2+c_2 \alpha^3+m}{\delta_3}\right)=\frac{4m}{\delta_3}\in \io{K}\label{eq-7}\end{equation}  Since \(m\) is odd, it follows that \(\delta_3\mid 4\).

Suppose \(\delta_3=4\). Then, arguing as above we get  
\[Tr_{N/K}\left(\nfrac{c_0\alpha+c_1\alpha^2+c_2\alpha^3+m}{4}\right)=m\in \io{K}\]  Since \(m\) is odd and \((m)\subset  Tr_{N/K}\left(\io{N}\right)\) it follows that \(1+\iota\) does not divide \(Tr_{N/K}(\io{N})\).  From \ref{tame-ramification}, it follows that \(1+\iota\) is unramified in \(N/K\).

Conversely, suppose  \((1+\iota)\) is unramified. By property \ref{I}, it follows that \(1+\iota\) doesn't divide \(f\), \(g\) or \(h\) and hence \(m\) is odd. From property \ref{III}, it follows that \(\nu_{1+\iota}(\delta_1)\leq 1\). Suppose \(\delta_3\mid 2(1+\iota)\), i.e. \(\nu_{1+\iota}(\delta_3)\leq 3\).   From \ref{main-equation}. we get \(\nu_{1+\iota}\left(\delta_1\delta_2\delta_3\right)=8\) since \(1+\iota\) doesn't divide \(\disc(N/K)\). We have \(\nu_{1+i}(\delta_2)\leq 3\) since \(\delta_2\mid \delta_3\). So, \(\nu_{1+\iota}(\delta_2\delta_2\delta_3)\leq 7\) which is a contradiction.  So, \(\delta_3=4\).
\end{enumerate}
\end{proof}
\begin{theorem}\label{integral-basis}
	An integral basis of $N$ over $K$ is as in \ref{main-table}.
	
	\begin{table}[htb]
		\noindent\centering
		\begin{tabular}{llc}
			\hline
			\multicolumn{1}{c}{Condition} & Integral Basis&\(\Delta\) \\
			\hline
			\rule[-0.4cm]{0pt}{1cm}	$m \equiv 1 \pmod8$ 
			&\(\left\{ {1},   \nfrac{1+\alpha}{1+\iota}, 
			{\nfrac{{|gh|^2(\iota+(1+\iota)\alpha)}+{\alpha^2}}{2(1+\iota)gh}}, 
			\nfrac{{|gh^2|}^2(1+\alpha+{\alpha^2})+{\alpha^3}} {4g{h^2}}\right\}\)&\\
			\hline
			\rule[-0.4cm]{0pt}{1cm}		$m \equiv {1+4\iota}\pmod8$
			& $\left\{ 1,   \nfrac{1+\alpha}{1+\iota}, 
			{\nfrac{{|gh|^2(-\iota+(1+\iota)\alpha)}+{\alpha^2}}{2(1+\iota)gh}}, 
			\nfrac{{|gh^2|}^2(2-\iota+\alpha+\iota{\alpha^2})+{\alpha^3}} {4g{h^2}}\right\}$&\\
			\hline 
			
			\hspace*{-2mm}$\begin{array}{l}
				{m\equiv 2\iota \pmod4}\\
				fh \equiv 1\pmod4 
			\end{array}$
			& $\left\{ 1,   \alpha, 
			\nfrac{gh+\alpha^2}{2gh}, 
			\nfrac{\alpha+\alpha^3/gh^2}{2}\right\}$&\\
			\hline
			
			\hspace*{-2mm}$\begin{array}{l}
				{m\equiv 2\iota \pmod4}\\
				fh \equiv -1\pmod4 
			\end{array}$ 
			& $\left\{ 1,   \alpha, 
			\nfrac{\iota gh+\alpha^2}{2gh}, 
			\nfrac{\iota\alpha+\alpha^3/gh^2}{2}\right\}$&\\
			\hline
			
			\hspace*{-2mm}$\begin{array}{l}
				{m\equiv 2\iota \pmod4}\\
				fh \equiv \pm1\pmod{2(1+\iota)}\\ 
				f\bar{h}\equiv 1 \pmod{2(1+\iota)}
			\end{array}$
			& $\left\{1, \alpha, 
			\nfrac{gh+\alpha^2}{(1+\iota)gh}, 
			\nfrac{\alpha+\alpha^3/gh^2}{2}\right\}$&\\
			\hline
			
			\hspace*{-2mm}$\begin{array}{l}
				{m\equiv 2\iota \pmod4}\\
				fh \equiv \pm1\pmod{2(1+\iota)}\\ 
				f\bar{h}\equiv -1\pmod{2(1+\iota)}
			\end{array}$
			& $\left\{ 1,   \alpha, 
			\nfrac{gh+\alpha^2}{(1+\iota)gh}, 
			\nfrac{\iota\alpha+\alpha^3/gh^2}{2}\right\}$&\\
			\hline
			\rule[-0.4cm]{0pt}{1cm}$m \equiv3+2\iota \pmod4$ 
			& $\left\{ 1, \alpha, \nfrac{|gh|^2+\alpha^2}{(1+\iota)gh}, 
			\nfrac{{|gh^2|^2(1+\alpha+\alpha^2)}+\alpha^3}{2gh^2}\right\}$&\\
			\hline
			\rule[-0.4cm]{0pt}{1cm}		$m \equiv 1+2\iota \pmod4$ 
			& $\left\{ 1,   \alpha,\nfrac{|gh|^2(1+(1+\iota)\alpha)+\alpha^2}{2gh}, 
			\nfrac{|gh^2|^2(\alpha+(1+\iota)\alpha^2)+\alpha^3}{2gh^2}\right\}$&\\
			\hline
			\rule[-0.4cm]{0pt}{1cm}		$m \equiv 3 \pmod4$ 
			& $\left\{ 1,  \alpha, 
			\nfrac{\iota|gh|^2+\alpha^2}{2gh}, 
			\nfrac{{|gh^2|^2}(\iota+\iota\alpha+\alpha^2)+\alpha^{3}}{2gh^2}\right\}$&\\
			\hline
			
			\hspace*{-2mm}$\begin{array}{l}
				{m \equiv 5 \pmod8}\\
				\mbox{or}\\
				{m \equiv 5+4\iota \pmod8} 
			\end{array}$
			& $\left\{ 1,  \nfrac{1+\alpha}{1+\iota}, 
			\nfrac{|gh|^2+\alpha^2}{2gh}, 
			\nfrac{{|gh^2|}^2(1+\alpha+{\alpha^2})+{\alpha^3}} {2(1+\iota)gh^2}\right\}$&\\
			\hline
			
			\hspace*{-2mm}$\begin{array}{l}
				f~\mbox{is even~~or}
				~~\mbox{h is even}\\ \mbox{or}\\
				{m \equiv \iota\pmod{2}} 
			\end{array}$
			& $\left\{ 1,    \alpha, 
			\nfrac{\alpha^2}{gh}, \nfrac{\alpha^3}{gh^2}\right\}\)&\\ \hline
			\rule[-0.4cm]{0pt}{1cm}$\begin{array}{l}fh\equiv\iota\pmod{2}\\ \mbox{and}~g\mbox{ is even} \end{array}$& $\left\{1 ,  \alpha,
			\nfrac{\alpha^{2}}{gh} \nfrac{\iota\alpha+\alpha^{3}/gh^{2}}{1+\iota}\right\}$&\\
			\hline
		\end{tabular}
		\caption{Integral Bases}\label{main-table}
	\end{table}
\end{theorem}\vspace{3mm}
We divide the proof into two cases, \(N/K\) is ramified at \(1+\iota\) and \(N/K\) is unramified at \(1+\iota\).
\begin{proposition} The extension \(N/K\) is unramified at \(1+\iota\) iff \(m\equiv 1\pmod{8}\) or \(m\equiv 1+4\iota\pmod{8}\). If \(m\equiv 1\pmod{8}\) an integral basis is  
	\[
	\left\{ {1},   \nfrac{1+\alpha}{1+\iota}, 
	{\nfrac{{|gh|^2(\iota+(1+\iota)\alpha)}+{\alpha^2}}{2(1+\iota)gh}}, 
	\nfrac{{|gh^2|}^2(1+\alpha+{\alpha^2})+{\alpha^3}} {4g{h^2}}\right\}
	\]
	If \(m\equiv 1+4\iota \pmod{8}\), an integral basis is 
	\[
	\left\{ 1,   \nfrac{1+\alpha}{1+\iota}, 
	{\nfrac{{|gh|^2(-\iota+(1+\iota)\alpha)}+{\alpha^2}}{2(1+\iota)gh}}, 
	\nfrac{{|gh^2|}^2(2-\iota+\alpha+\iota{\alpha^2})+{\alpha^3}} {4g{h^2}}\right\}
	\]Further, the discriminant of \(N/K\) is \(f^3g^2h^3\).
	\end{proposition}
\begin{proof}
		Consider the extension $N/K$. From \ref{prop-3}, property \ref{p-0} and property \ref{p-2}, we know the odd parts of \(d_1\), \(d_2\) and \(d_3\). So, we need to determine  $\delta_{1},\delta_{2},\delta_{3}$, the even parts of $d_{1},d_{2},d_{3}$.
	
	Since $N/K$ is unramified at \(1+\iota\), from \ref{prop-3}, property \ref{I}, $m$ is odd. From \ref{prop-3}, property \ref{II}, we also have \(\delta_3=4\).
	
	If \(\nu_{1+\iota}(\delta_2)<3\), since \(\nu_{1+\iota}(\delta_1)\leq 1\)
	we get \(\nu_{1+\iota}(\delta_1\delta_2\delta_3)\leq 7\) which contradicts \ref{main-equation}.  So, \begin{equation}\label{eq-8}\nu_{1+\iota}(\delta_2)\geq 3\end{equation}
	
	We have 
	\[
	gh\left(\sigma^{2}(\alpha_3)-\alpha_3\right)=\frac{2b_1}{\delta_2}\in \io{K}
	\]
	Therefore \(\nu_{1+\iota}(b_1)\geq \nu_{1+\iota}(\delta_2)-\nu_{1+\iota}(2)\geq 1\).  We claim that \(\nu_{1+\iota}(b_1)=1\).  Assume for contradiction that \(2\mid b_1\). Considering
	\begin{equation}\label{eq-9}
		Tr_{M/K}(\alpha^{2}N_{N/M}(\beta_{3}))=\frac{m(2b_0-b_1^{2})}
		{\delta_{2}^{2}}\in \io{K}
	\end{equation} 
	since \(\nu_{1+\iota}\left(\delta_3\right)\geq 3\), we get \(2b_0-b_1^2\equiv 0\pmod{4}\).  It follows that \(b_0\equiv 0\pmod{2}\).  From \(Tr_{N/M}(\alpha_2)=2\left(\frac{b_0+\alpha^2}{\delta_2}\right)\in\io{M}\)  
	it follows that \(1+\iota \mid \alpha^2\). So, \(\alpha^2\equiv 0\pmod{1+
		\iota}\) and therefore \(1+\iota \mid m\).  So, by \ref{prop-3}, \vref{I}, it follows that 
	\(1+\iota\) is ramified in \(N/K\) which is  a
	contradiction to our assumption that $N/K$ is unramified at $1+\iota$.  Therefore, \(\nu_{1+\iota}(b_1)=1\).
	
	Since \(\nu_{1+\iota}(b_1)=1\),  from \(\nu_{1+\iota}\left(\nfrac{2b_1}{\delta_2}\right)\geq 0\) it follows that 
	\(\nu_{1+\iota}(\delta_2)\leq 3\). So, from \ref{eq-8}, it follows that \(\nu_{1+\iota}(\delta_2)=3\).  Since \(\delta_3=4\), from \ref{main-equation}, it follows that \(\delta_1=1+\iota\). From \ref{prop-3}, property \ref{III}, it follows that \(m\equiv 1\mod{4}\).  
	
	We now determine the possible values of \(b_0\) and \(b_1\).  Note that, these are determined only modulo \(2(1+\iota)\). We have $\delta_2=2(1+\iota)$ and \(1+\iota \nmid m\) since \(m\equiv 1\pmod{4}\). So, from \ref{eq-9} it follows that
	$2b_0-b_1^{2}\equiv 0\pmod{8}$. Also, \(b_1=a(1+i)\) with \(a\) odd. So, \(b_1^2=2\iota a^2\equiv \pm 2i\pmod{8}\) since \(a^2\equiv \pm 1,\pm 3+4\iota\pmod{8}\). This implies that \(2b_0\equiv \pm 2i\pmod{8}\) or 
	\(b_0\equiv \pm \iota \pmod{4}\). Also \(b_1^2 \equiv  2\iota\pmod{2(1+\iota)}\), so \(b_1\equiv 1\pm \iota\pmod{2(1+\iota)}\).
	So we can take the value of $b_0$ and $b_1$ to be by $\pm
	\iota$ and $1\pm \iota$ respectively.  
	
	We show next that \(m\equiv 1\pmod{8}\) or \(m\equiv 1+4\iota\pmod{8}\) if \(N/M\) is unramified at \(1+\iota\). We have $\alpha_3={\displaystyle 
		b_0+b_1\alpha+\alpha^{2}\over
		{\displaystyle d_{2}}}={\displaystyle 
		b_0+b_1\alpha+\alpha^{2}\over
		{\displaystyle \delta_{2}gh}}\in\io{N}$.  Therefore, \(\beta_3=gh\alpha_3={\displaystyle 
		b_0+b_1\alpha+\alpha^{2}\over
		{\displaystyle \delta_{2}}}\in\io{N}\). 

	Consider
	\[
	N_{N/M}(\beta_3)=\nfrac{b_0^{2}+m+\alpha^{2}(2b_0-b_1^{2})}{8}\in\io{M}
	\]
	
	\noindent Since $b_0=\pm \iota$ and $b_1=1\pm \iota$, $2b_0-b_1^{2}\equiv 0\hbox{ or
	}4\iota\pmod{8}$ and $b_0^{2}+m=m-1$. When $2b_0-b_1^{2}\equiv 0\pmod{8}$,
	$\nfrac{\alpha^2(m-1)}{8}\in \io{M}$ or $m\equiv 1\pmod{8}$. When $2b_0-b_1^{2}\equiv
	4\iota\pmod{8}$, \[\nfrac{m-1+4\iota\alpha^{2}} {8}=\nfrac{m-(1+4\iota)+4\iota(1+\alpha^2)}{8}\in\io{M}\] is an integer. We have  \(\nfrac{4\iota(1+\alpha^2)}{8}=\frac{\iota\left(1+\alpha^2\right)}{2}\) is an integer. (Considering \(N_{M/K}\left(\nfrac{1+\alpha^2}{2}\right)\) and \(Tr_{M/K}\left(\nfrac{1+\alpha^2}{2}\right)\) we can easily check that \(\nfrac{1+\alpha^2}{2}\) is an integer.) So, \(\nfrac{m-1+4\iota}{8}\in \io{M}\) or 
	$m\equiv 1+4\iota\pmod{8}$. So if $N/K$ is unramified at \(1+\iota\), $m\equiv 1\hbox{
		or }1+4\iota\pmod{8}$.
	
	Conversely, if $m\equiv 1\hbox{ or }1+4\iota\pmod{8}$, \(m\) is odd and  $\delta_{1}=1+\iota$.
	If $m\equiv 1\pmod{8}$, we claim that 
	$\nfrac{\iota+(1+\iota)\alpha+\alpha^{2}}{2(1+\iota)}\in\io{N}$.
	Since \(N/M\) is a quadratic extension, it is enough to show that  \[Tr_{N/M}\left(\nfrac{\iota+(1+\iota)\alpha+\alpha^{2}}{2(1+\iota)}\right)\in \io{M}\textrm{ and }N_{N/M}\left(\nfrac{\iota+(1+\iota)\alpha+\alpha^{2}}{2(1+\iota)}\right)\in\io{M}\]
	We have 
	\(Tr_{N/M}\left(\nfrac{\iota+(1+\iota)\alpha+\alpha^{2}}{2(1+\iota)}\right)=\nfrac{\iota +\alpha^2}{1+\iota}\). To show that   
	\(\left(\nfrac{\iota +\alpha^2}{1+\iota}\right)\in\io{M}\) is an integer, we again look at its norm and trace over \(M/K\). Here,
	\[N_{M/K}
	\left(\nfrac{\iota +\alpha^2}{1+\iota}\right)=\nfrac{-1-m}{2i}\in\io{K},Tr_{M/K}\left(\frac{\iota+\alpha^2}{1+\iota}\right)=\frac{2\iota}{1+\iota}\in \io{K}\] since \(m\equiv 1\pmod{4}\).
	
	Using \ref{beta-4-nr} we have \begin{align*}N_{N/M}\left(\nfrac{\iota+(1+\iota)\alpha+\alpha^{2}}{2(1+\iota)}\right)&=\nfrac{(-1+m)+(2\iota -(1+\iota)^2)\alpha^2}{8}\\
		&=\nfrac{m-1}{8}\in \io{M}
	\end{align*}
	So
	$2(1+\iota)\vert \delta_{2}$. Thus,
	$4\vert\delta_{1}\delta_{2}\vert\delta_{3}$ and $\delta_{3}=4$. From \vref{main-equation}, it follows
	that $\delta_{2}=2(1+\iota)$ since \(f\), \(g\) and \(h\) are odd. If $m\equiv 1+4\iota\pmod{8}$, we have \[\frac{-\iota+(1+\iota)\alpha)+\alpha^2}{2(1+\iota)}\in \io{N}\] As before, \(\delta_1=1+\iota\), \(\delta_2=2(1+\iota)\) and \(\delta_3=4\). Again, from \vref{main-equation}, we get \(\disc(N/M)=f^3g^2h^3\) and  $N/K$ is unramified at \(1+\iota\) in this case.
	
	To prove that \[\alpha_3={\nfrac{{|gh|^2(\iota+(1+\iota)\alpha)}+{\alpha^2}}{2(1+\iota)gh}}\]
	is an integer, we need to show that 
	\[
	\nu_Q\left(|gh|^2(\iota+(1+\iota)\alpha)+{\alpha^2}\right)\geq \nu_Q(2(1+\iota)gh)
	\]
	for each prime ideal \(Q\) in \(\io{N}\) dividing \(2(1+\iota)gh\). This is clear for \(Q\mid gh\) once we write the numerator as \(gh\overline{gh}(\iota +(1+\iota)\alpha)+gh\frac{\alpha^2}{gh}\) since \(\pi\neq 1+\iota\).
	
	For \(\Q\) in \(\io{N}\), \(Q\mid 1+\iota\),  since \(\nu_{1+\iota}(gh)=0\) we need to show that 
	\[
	\nu_Q\left(|gh|^2(\iota+(1+\iota)\alpha)+{\alpha^2}\right)\geq \nu_Q(2(1+\iota))
	\] or
	\(\nfrac{|gh|^2(\iota+(1+\iota)\alpha)+{\alpha^2}}{2(1+\iota)}\) is an integer.  
	Since \(\abs{gh}^2\equiv 1\pmod{4}\), it is enough to show that 
	\(\nfrac{(\iota+(1+\iota)\alpha)+{\alpha^2}}{2(1+\iota)}\) is an integer and we are done.
	
	Similarly, we can show that \(\nfrac{\abs{gh}^2\left(-\iota+(1+\iota)\alpha\right)+\alpha^{2}}
	{2(1+\iota)}\) is an integer if \(m\equiv 1+4\iota\pmod{8}\).
	
	\noindent Assume that \(m\equiv 1+4\iota\pmod{8}\). Note that \(\frac{\alpha^3}{gh^2}\in \io{N}\) it satisfies the polynomial 
	\(x^4-f^3g^2h=0\). As before, to show that \[\nfrac{{|gh^2|}^2(2-\iota+\alpha+\iota{\alpha^2})+{\alpha^3}} {4g{h^2}}=\nfrac{gh^2\overline{gh^2}(2-\iota+\alpha+\iota{\alpha^2})+{gh^2\frac{\alpha^3}{gh^2}}} {4g{h^2}}\] is an integer, it is enough to show that \(\nfrac{2-\iota+\alpha+\iota\alpha^2+\alpha^3}{4}\) is an integer.
	We have
	\[
	Tr_{N/M}\left(\nfrac{2-\iota+\alpha+\iota\alpha^2+\alpha^3}{4}\right)=\nfrac{2-\iota+\iota\alpha^2}{2}=\nfrac{2+\iota(\alpha^2-1)}{2}\in \io{M}
	\]	
	because we have already proved that \(\nfrac{\alpha^2-1}{2}\in\io{M}\) if \(m\equiv 1\pmod{4}\).
	
	From \ref{beta-4-nr} we have 
	\begin{align*}
		N_{N/M}\left(\nfrac{2-\iota+\alpha+\iota\alpha^2+\alpha^3}{4}\right)&= \nfrac{\left((2-\iota)^2-m-2m\right)+\left(4\iota +1-m\right)\alpha^2}{16}\\
		&=\nfrac{3-4\iota -3m+\left(4\iota +1 -m\right)\alpha^2}{16}
		\intertext{Writing \(m=1+4\iota+8\mu\) with \(\mu \in \io{K}\), we get }
		N_{N/M}\left(\nfrac{2-\iota+\alpha+\iota\alpha^2+\alpha^3}{4}\right)&=\nfrac{-16\iota-24\mu-8\mu\alpha^2}{16}=\frac{-2\iota-3\mu-\mu\alpha^2}{2}
	\end{align*}
	We have 
	\begin{align*}
		N_{M/K}\left(\frac{-2\iota+3\mu-\mu\alpha^2}{2}\right)&=\nfrac{-4+(9-m)\mu^2+12\iota\mu}{4}\in \io{K} 
		\intertext{since \(9-m=8-4\iota-8\mu\). Also,} Tr_{M/K}\left(\frac{-2\iota+3\mu-\mu\alpha^2}{2}\right)&=-2\iota+3\mu\in \io{K}
	\end{align*}
		\end{proof}
\begin{proposition} 
	If \(N/K\) is ramified at \(1+\iota \), the integral basis is as in \ref{main-table-1}.
	\begin{table}[htb]
		\noindent\centering
		\begin{tabular}{llc}
			\hline
			\multicolumn{1}{c}{Condition} & Integral Basis&\(\Delta\) \\
			\hline
			\hspace*{-2mm}$\begin{array}{l}
				{m\equiv 2\iota \pmod4}\\
				fh \equiv 1\pmod4 
			\end{array}$
			& $\left\{ 1,   \alpha, 
			\nfrac{gh+\alpha^2}{2gh}, 
			\nfrac{gh^2\alpha+\alpha^3}{2gh^2}\right\}$&\\
			\hline
			
			\hspace*{-2mm}$\begin{array}{l}
				{m\equiv 2\iota \pmod4}\\
				fh \equiv -1\pmod4 
			\end{array}$ 
			& $\left\{ 1,   \alpha, 
			\nfrac{\iota gh+\alpha^2}{2gh}, 
			\nfrac{gh^2\iota\alpha+\alpha^3}{2gh^2}\right\}$&\\
			\hline
			
			\hspace*{-2mm}$\begin{array}{l}
				{m\equiv 2\iota \pmod4}\\
				fh \equiv \pm1\pmod{2(1+\iota)}\\ 
				f\bar{h}\equiv 1 \pmod{2(1+\iota)}
			\end{array}$
			& $\left\{1, \alpha, 
			\nfrac{gh+\alpha^2}{(1+\iota)gh}, 
			\nfrac{gh^2\alpha+\alpha^3}{2gh^2}\right\}$&\\
			\hline
			
			\hspace*{-2mm}$\begin{array}{l}
				{m\equiv 2\iota \pmod4}\\
				fh \equiv \pm1\pmod{2(1+\iota)}\\ 
				f\bar{h}\equiv -1\pmod{2(1+\iota)}
			\end{array}$
			& $\left\{ 1,   \alpha, 
			\nfrac{gh+\alpha^2}{(1+\iota)gh}, 
			\nfrac{gh^2\iota\alpha+\alpha^3}{2gh^2}\right\}$&\\
			\hline
			\rule[-0.4cm]{0pt}{1cm}$m \equiv3+2\iota \pmod4$ 
			& $\left\{ 1, \alpha, \nfrac{|gh|^2+\alpha^2}{(1+\iota)gh}, 
			\nfrac{{|gh^2|^2(1+\alpha+\alpha^2)}+\alpha^3}{2gh^2}\right\}$&\\
			\hline
			\rule[-0.4cm]{0pt}{1cm}		$m \equiv 1+2\iota \pmod4$ 
			& $\left\{ 1,   \alpha,\nfrac{|gh|^2(1+(1+\iota)\alpha)+\alpha^2}{2gh}, 
			\nfrac{|gh^2|^2(\alpha+(1+\iota)\alpha^2)+\alpha^3}{2gh^2}\right\}$&\\
			\hline
			\rule[-0.4cm]{0pt}{1cm}		$m \equiv 3 \pmod4$ 
			& $\left\{ 1,  \alpha, 
			\nfrac{\iota|gh|^2+\alpha^2}{2gh}, 
			\nfrac{{|gh^2|^2}(\iota+\iota\alpha+\alpha^2)+\alpha^{3}}{2gh^2}\right\}$&\\
			\hline
			
			\hspace*{-2mm}$\begin{array}{l}
				{m \equiv 5 \pmod8}\\
				\mbox{or}\\
				{m \equiv 5+4\iota \pmod8} 
			\end{array}$
			& $\left\{ 1,  \nfrac{1+\alpha}{1+\iota}, 
			\nfrac{|gh|^2+\alpha^2}{2gh}, 
			\nfrac{{|gh^2|}^2(1+\alpha+{\alpha^2})+{\alpha^3}} {2(1+\iota)gh^2}\right\}$&\\
			\hline
			
			\hspace*{-2mm}$\begin{array}{l}
				f~\mbox{is even~~or}
				~~\mbox{h is even}\\ \mbox{or}\\
				{m \equiv \iota\pmod{2}} 
			\end{array}$
			& $\left\{ 1,    \alpha, 
			\nfrac{\alpha^2}{gh}, \nfrac{\alpha^3}{gh^2}\right\}\)&\\ \hline
			\rule[-0.4cm]{0pt}{1cm}$\begin{array}{l}fh\equiv\iota\pmod{2}\\ \mbox{and}~g\mbox{ is even} \end{array}$& $\left\{1 ,  \alpha,
			\nfrac{\alpha^{2}}{gh}, \nfrac{gh^2\iota\alpha+\alpha^{3}}{(1+\iota)gh^2}\right\}$&\\
			\hline
		\end{tabular}
		\caption{Integral Bases}\label{main-table-1}
	\end{table} 
	\end{proposition}
\begin{proof}

	\textbf{Case 1}
	$N/M$ is ramified and $M/K$ is unramified at \(1+\iota\). Note that $fh\equiv \pm
	1\pmod{4}$ since $M/K$ is unramified at \(1+\iota\). 
	
	\noindent
	\textbf{Case 1(a)} $f,g,h$ are odd.
	
	Since $m$ is odd and $m\equiv\pm 1\pmod{4}$, $m\equiv 5 {\rm ~or~} 5+4\iota\pmod{8}$
	or $m\equiv 3\pmod{4}$. When $m\equiv 5 {\rm ~or~} 5+4\iota\pmod{8}$,
	$\delta_{1}=1+\iota$ and so $2\vert \delta_{2}$ and
	$2(1+\iota)\vert\delta_{3}$. Since $M/K$ is ramified at $1+\iota$, from \ref{prop-3},  \ref{II}, it
	follows that $\delta_{3}\neq 4$ and $\delta_{3}=2(1+\iota)$. As
	$\delta_{1}\delta_{2}\vert \delta_{3}$, $2(1+\iota)\not\vert\delta_{2}$.
	Thus, $\delta_{2}=2$.
	
	Let $m\equiv 3\pmod{4}$. Then $\delta_{1}=1$. We have $N_{N/K}(\alpha-1)=m-1\equiv
	2\pmod{4}$. It follows that \((1+\iota)^2\parallel N_{N/K}(\alpha-1)\). We claim that \(\nu_Q(\alpha-1)=1\) for all primes in \(\io{N}\) with \(Q\cap \io{K}=(1+\iota)\). Primes above $1+\iota$  are \emph{unramified} in \(M \slash K\) and (totally) ramified in \(N/M\).  So, the possibilities are 
	\(f(Q\vert (1+\iota))=1\) or \(2\). 
	If \(f(Q\vert (1+\iota))=1\), there are primes \(Q_1\) and \(Q_2\) in \(N\) dividing \(1+\iota\) such that 
	\((1+\iota)=Q_1^2Q_2^2\). Further, if \(Q_1\vert (\alpha-1)\), \(Q_2=\sigma^2(Q_1)\) divides \(-(\alpha+1)=(\alpha+1)\).  So, \(Q_2\) divides 
	\((\alpha-1)\) since \((\alpha-1)\subset (\alpha+1)+(2)\subset Q_2\). So, \(Q_1\) divides \(\alpha-1\) iff \(Q_2\) divides \(\alpha-1\).  Writing \((\alpha-1)=Q_1^{k_1}Q_2^{k_2}I\), where \(Q_1\) and \(Q_2\) do not divide \(I\), we have \(N_{N/K}(\alpha-1)=(1+\iota)^{k_1+k_2}N_{N/K}(I)\).  Since \(2\parallel (\alpha-1)\), we have \(k_1+k_2=2\), so \(k_1=1\) and \(k_2=1\). If \(f(Q\vert(1+\iota))=2\), we have \(Q^2=(1+\iota)\). Suppose  \((\alpha-1)=Q^kI\) where \(Q\) does not divide \(I\). We have \(N_{N/K}(\alpha-1)=(1+\iota)^{2k}N_{N/K}(I)\) and we get \(k=1\).    
		
Since,	$\alpha -1$ is a uniformiser at all primes $Q$ in $N$ dividing
	$1+\iota$, it follows that \(\nu_{Q}(\alpha-1)=1\) and \(\alpha-1\) is a uniformiser for the unique prime ideal in \(\io{N,Q}\), the completion of \(N\) at \(Q\). This is because the discrete valuation on \(N_Q\), the completion of \(N\) at \(Q\) is the unique extension of the discrete valuation on \(N\) defined by the prime ideal \(Q\).   
	
	Let \(Q_M=P\cap \io{M}\) be a prime in \(\io{M}\). Let \(M_{Q_M}\) be the completion of \(M\) with respect to the discrete valuation defined by \(Q_M\).  Since \(N_Q/M_{Q_M}\) is totally ramified, it follows that \(\left\{1,\alpha-1\right\}\) generates \(\io{N_Q}\) over \(\io{M_{Q_M}}\). So, we have 
	\[
	G_i=\left\{\sigma\in G(N_Q/M_{Q_M})\vert \sigma(\alpha-1)-(\alpha-1)\in Q^{i+1}\right\}
	\] 
	Suppose \(1+\iota\) is inert in \(M\).  We have $\sigma^{2}(\alpha-1)-(\alpha -1)=-2\alpha\in Q^{4}$
	when $1+\iota$ is inert in $M$ and $Q$ is the unique prime in $N$
	dividing $1+\iota$. So $\sigma^{2}\in G_{3}(F/L)$ and $\sigma^{2}\not\in
	G_{4}(F/L)$. Using the formula for the different(cf. \cite{lf})
	$$ {v_Q}({{\cal D}\left(N_Q/M_{Q_M}\right)}) = \sum_{i=0}^{\infty} ({|G_{i}|-1})$$
	we get
	$\nu_Q({\cal D})=4$.   Therefore \[d\left(N_Q/M_{Q_M}\right)=N_{N_Q/M_{Q_M}}\left({\cal D}\right)=N_{N_Q/M_{Q_M}}\left(Q^4\right)=(1+\iota)^4\]
	Therefore \((1+\iota)^4\parallel d\left(N/M\right)\). Since \(d(N/K)=N_{N/M}(d(N/M))\left(d(M/K)\right)^{[N:M]}\) and \((1+\iota)\nmid d(M/K)\), it follows that \((1+\iota)^8 \parallel d(N/K)\).
	
	 If $1+\iota$ splits in $M$ and
	$Q$ is one of the two primes in $N$ dividing $1+\iota$,
	$Q^{4}\parallel 2\alpha$. From this, it follows that for any prime
	$Q$ in $N$ dividing $1+\iota$, we have $Q_M^4\parallel 
	d(N/M)$ where \(Q_M=Q\cap \io{M}\). If \(Q_M\) and \(Q_M'\) are the two primes in \(M\) above \((1+\iota)\),
	$\left(Q_MQ_M'\right)^4=(1+\iota)^{4}\parallel \disc(N/M)$. As before, \((1+\iota)^8\parallel d(N/K)\).
	
	 Therefore, from
	(\ref{main-theorem}), it follows that $\delta_{2}\delta_{3}=4$. Since
	$${\iota {\abs{gh}}^{2}+\alpha^{2}\over
		2gh},{\abs{gh^{2}}^{2}\iota\alpha+\alpha^{3}\over 2gh^{2}}\in\io{F},$$
	$\delta_{2}=\delta_{3}=2$.
	
	\noindent
	\textbf{Case 2(b)} Suppose $g$ is even. Let \(Q\) be a prime ideal in \(N\) that divides \(1+\iota\). Since \((1+\iota)\mid \alpha^4=fg^2h^3\), \(Q\mid (\alpha)\). Let \((\alpha)=Q^rI\) where \(Q\nmid I\). We have \(N_{N/M}(\alpha)=N_{N/M}\left(Q^r\right)N_{N/M}(I)=Q_M^rJ\) where \(Q_M=\io{M}\cap Q\), \(Q_M\nmid J\). Then, since \(f(Q\vert Q_M)=1\),  
	\begin{equation}v_{Q_{M}}(N_{F/M}(\alpha))=r\end{equation}
	where $Q_{M}$ is a prime in $M$ dividing $1+\iota$. On the other hand \(1+\iota\) is unramified in \(L\), so \(\nu_{Q_M}((1+\iota))=1\). Since \((1+\iota)^2\nmid g\), it follows that $\nu_{Q_M}(gh\sqrt{fh})=1$.  Therefore, \(\nu_{Q_M}(N_{N/M}(\alpha))=\nu_{Q_M}(gh\sqrt{fh})=1\).  So. \((\alpha)=QI\) and thereforee $\alpha$ is a 
	uniformiser at all primes in $F$ dividing $1+\iota$.
	
	We have $\sigma^{2}(\alpha)-\alpha = 2\alpha$. As in the previous case we can
	check that $(1+\iota)^{10}\parallel\disc(N/K)$ and
	$\delta_{2}\delta_{3}=(1+\iota)^{6}$ . Since
	$$ {gh+\alpha^{2}\over 2gh},{igh^{2}+\alpha^{3}\over 2gh^{2}}\in\io{F}$$
	\(2\mid \delta_2\), \(2\mid \delta_3\) and \(\delta_3\neq 4\) we must have $\delta_{2}=\delta_{3}=2(1+\iota)$.
	
	\noindent
	\textbf{Case 3} $N/K$ is totally ramified at $1+\iota$. Note that $N/K$ is
	totally ramified at $1+\iota$ if and only if $1+\iota$ ramifies in $M/K$.
		
Let $G_i$, $i\geq 0$, be the ramification groups of \(N/K\). Since \(1+\iota\) is totally ramified in $N/K$, these are the same as the ramification groups of $N'/K'$  where $K'$ is the completion of \(K\) at \(1+\iota\) and $N'$ is the completion of $N$ at the unique prime in $N$ that lies over $1+\iota$ in \(N\). 
	
We will use the results from \cite{wyman}. Let $b_{i}$ and $b^{i}$ denote
the $i^{th}$ lower and upper break numbers of the extension $M/K$. 

\b{Case 3(a)}/ $m\equiv \iota\pmod{2}$ and $g$ is \emp{odd}/.

In this case, $4\parallel\disc (M/K)$. From the formula for the different,
it follows that the break number for the extension for $M/K$ is $3$. So, from \ref{prop:7a} 
it follows that $b^{1} =b_{1}=3$ for the extension for $N/K$.  
Using   \ref{wyman}, it follows that $b^{2}=5$.  From \ref{eqn:10-a} it follows that 
$b_{2}= 7$ for the extension $N/K$. Using the formula for the different,
we get that $(1+\iota)^{16}\parallel \disc (N/K)$. From (\ref{main-equation}), it follows that
$\delta_{1}\delta_{2}\delta_{3}=1$. 
	
\noindent	\textbf{Case 3(b)} $fh\equiv \iota\pmod{2}$, $g$ is \emp{even}/.
	
	As before, $(1+\iota)^{16}\parallel \disc (N/K)$. So we get, from (\ref{main-equation}),
	\[
	\left(d_1d_2d_3\right)^2=(1+\iota)^6f^3g_1^6h^9
	\] where \(g_1\) is the odd part of \(g\). Therefore, 
	$\delta_{2}\delta_{3}=(1+\iota )^{3}$. Since $\alpha^{2}/gh\in\io{F}$ and
	$\delta_{2}\vert \delta_{3}$, the only possibility is $\delta_{2}=1+\iota$
	and $\delta_{3}= 2$.
	
	\noindent
	\textbf{Case 3(c)} $f$ or $h$ is divisible by $1+\iota$. 
	
	In either case $(1+\iota)^{5}\parallel\disc (M/K)$. As in the previous case, we
	have $b^{1}(M/K)=b^1(N/K)=4$, $b^{2}= 6$ and $b_{2}=8$. So using the
	formula for the different, we get $(1+\iota)^{19}\parallel  
	\disc (N/K)$. 
	
	When $f$ is even, $\delta_{2}\delta_{3}=1$. When $h$ is even,
	$\delta_{2}\delta_{3} =(1+\iota)^{3}$. As $\alpha^{2}/gh\in\io{F}$,
	$\delta_{2}=1+ \iota$ and $\delta_{3}= 2$.

	\noindent\textbf{Case 3(d)} $fh\equiv \pm 1+2\iota\pmod{4}$ and $g$ is odd.
	Since $2\parallel \disc (M/K)$ when $fh\equiv \pm 1+2\iota\pmod{4}$, the break 
	number is 1 for $M/K$  and so $b_{1}=1$ for $N/K.$ Using \ref{prop:8a}, we get that the break numbers are odd. Therefore, $b_{2}\geq 3$.
	$(1+\iota)^{8}\vert \disc (F/L)$. Thus, $\delta_{2}\delta_{3}\vert
	(1+\iota)^{4}$.
	
	When $m\equiv 1+2\iota\pmod{4}$, $\rat{1+(1+\iota)\alpha+\alpha^{2}}. {2}. \in\io{F}$.
	Therefore $\delta_{2}=2$, $\delta_{3}=2$.
	
	When $m\equiv 3+2\iota\pmod{4}$, 
	$$N_{N/K}(\alpha-1)=m-1\equiv 2(1+\iota)\pmod{4}$$
	So $\rat{(\alpha-1)^{3}}. {2}. $ is a uniformiser for the unique prime in $F$ dividing
	$1+\iota$. We have 
	$$\sigma^{2}\left(\rat{(\alpha-1)^{3}}. {2}.\right) - (\alpha-1)^{3}
	=\alpha(\alpha^{2}+3)$$
	and
	$N_{N/K}(\alpha^{2}+3)=(m-9)^{2}$. Since $m-9\equiv 2(1+\iota)\pmod{4}$, 
	$(1+\iota)^{6}\parallel N_{N/K}\alpha(\alpha^{2}+3)$ and so $\sigma^{2}\in G_{5}(N/K)$ and
	$\sigma^{2}\not\in G_{6}(N/K)$. So, in this case,  $b_{2}=5$ for $N/K.$  Thus,
	$(1+\iota)^{10}\vert\disc (N/K)$. Therefore,
	$\delta_{2}\delta_{3}\parallel (1+\iota)^{6}$. 
	So 
	$\delta_{2}\delta_{3}\vert (1+\iota)^{3}.$ 
	Since $\rat{{\abs{gh}}^{2}+
		\alpha^{2}}. {1+\iota}.$   
	is an  integer,
	$\delta_{2}=(1+\iota)$, $\delta_{3}=2.$
	
	\noindent
	\b{Case3(e)}/ Let $fh\equiv \pm 1+2\iota\pmod{4}$, $g$ be even.
	
	Since $fh\equiv\pm 1\pmod{2(1+i)}$, $f\equiv \pm h\pmod{2(1+\iota)}.$
	When $f\equiv h\pmod{2(1+\iota)}$, consider the element ${\cal X}=1+\rat{\alpha
		+\alpha^{3}/gh^{2}}. {2}.$ . We have 
	$$N_{N/K}({\cal X})= {16-(g(f-h))^{2}-16fgh\over 16}$$
	Since $g\equiv (1+\iota)\pmod{2}$, $g=a+\iota b$ with $a$ and $b$ odd. So
	$g^{2}\equiv\pm 2\iota\pmod{8}$. One checks easily that
	$(g(f-h))^{2}\equiv 16\pmod{32}$ and $16fgh\equiv 16(1+\iota)\pmod{32}$. 
	So ${\cal X}$ is a uniformiser. We have
	\begin{eqnarray*}
		\sigma^{2}({\cal X})-{\cal X}&=&\alpha+\alcghs\\
		\noalign{and}
		(1+\iota)^{8}\parallel N_{N/K}(\alpha+\alcghs)&=&g^{2}fh(f-h)^{2}.
	\end{eqnarray*}
	So
	$G_{7}\neq \{1\}$. Since $b_{1}$ is odd, so is $b_{2}$. So $b_{2}\neq 8$.
	Since $G_{9}(N/K)=1$, $b_{2}=7$. $(G_{i}(N/K)=\{1\}$ if $i > e/(p-1)$ for any
	extension of local fields $N/K$, where $p$ is the characteristic of the
	residue field and $e$ is the valuation of $p$ in $L$. Cf. \cite{lf},
	Exercise (2)c at the end of $\S$2 in Chapter 4.)
	So $(1+\iota)^{12}\vert \disc (N/K)$ and therefore
	$\delta_{2}\delta_{3}=(1+\iota)^{5}$. Since $\rat{gh+\alpha^{2}}. {gh}.$ and
	$\rat{\alpha+ \alpha^{3}/gh^{2}}. {2}.$ are in $\io{F}$, $\delta_{2}=2$,
	$\delta_{3}= 2(1+\iota)$.
	
	Similarly, when $f\equiv -h\pmod{2(1+\iota)}$,
	$1+\rat{\iota\alpha+\alpha^{3}/gh^{2} }. {2}.$ is an uniformiser. As
	before, it can be checked that $(1+\iota)^{12}\vert \disc (N/K)$. Since
	$\rat{gh+\alpha^{2}}. {(1+\iota)gh}.$ and
	$\rat{\iota\alpha+ \alpha^{3}/gh^{2}}. {2}.$ are in $\io{F}$, $\delta_{2}=2$,
	$\delta_{3}= 2(1+\iota)$.
\end{proof}

\section{Appendix}
In this appendix, we give the proof of main theorem on existence of normalised integral basis for the ring of integers of an extension field when the ring of integers for the base field is a PID.  The proof in \cite{mar} for the case where the base field is \Q\ goes through for this case also and no new ideas are needed.  We have reproduced adapted  proof here for  the sake of completeness.

We prove another lemma that will be useful in the proving the main theorem.
\begin{lemma}\label{key-lemma-1}
	Let  \(\theta\in \C\) have degree \(n\) over \(K\) and \(h_1(x)\), \(h_2(x)\in K[x]\) have degree less than \(n\) with \(h_1(\theta)=h_2(\theta)\). Then, \(h_1(x)=h_2(x)\).
\end{lemma}
\begin{proof}
	This follows immediately from the fact that the set  \(\{1,\theta,\ldots,\theta^{n-1}\}\) is linearly independet over \(K\)
\end{proof}
\begin{proof}[Proof of the  main theorem]
	Let \(\alpha\in S\) be such that \(L=K[\alpha]\).  Then \(\left\{1,\alpha,\alpha^2,\ldots,\alpha^{n-1}\right\}\) is a basis for \(L\) over \(K\). By, \ref{key-lemma}, we have 
	\[
	S\subseteq R\frac{1}{d}\oplus R\frac{\alpha}{d}\oplus\cdots\oplus R\frac{\alpha^{n-1}}{d}
	\] where
	\[d=\disc\left(1,\alpha,\ldots,\alpha^{n-1}\right)\]
	Let 
	\[
	F_i=R\frac{1}{d}\oplus R\frac{\alpha}{d}\oplus R\frac{\alpha^2}{d}\cdots \oplus R\frac{\alpha^{i-1}}{d}
	\]
	be the \(R\)-module generated by 
	\[
	\left\{\frac{1}{d},\frac{\alpha}{d},\frac{\alpha^2}{d},\ldots,\frac{\alpha^{i-1}}{d}\right\}
	\] and let \(S_i=F_i\cap S\). Each \(S_i\) is an \(R\)-module.  We have \(S_1=R\) and \(S_n=S\).  Let us see why. We have \(S_1\subset K\). Since elements in \(S_1\) are integers, \(R\) is integrally closed, and \(1\in S_1\) it follows that \(S_1=R\). Since \(S\subset F_n\) by \ref{key-lemma}, it follows that 
	\(S_n=S\).
	
	We define $f_k(x)\in R[x]$ and $d_k\in R$ such that for each $ i$ 
	\[
	\left\{1, \frac{f_1(\alpha)}{d_1},\ldots,\frac{f_{i-1}(\alpha)}{d_{i-1} } \right\}
	\] is a basis for $S_i$ and $d_{k}\mid d_{k+1}$ if $k\leq i-2$.  This is certainly true for $i=1$. 
	
	Suppose this is true for $m$, that is 
	\[
	S_{m}=\left\{1,\frac{f_1(\alpha)}{d_1},\ldots,\frac{f_{m-1}(\alpha)}{d_{m-1}}\right\}
	\]               
	is a basis for \(S_m\) for monic polynomials \(f_k(x)\) of degree \(k\) and \(d_k\in R\) such that \(d_k\mid d_{k+1}\) if \(k\leq m-2\).  We have to show that there is a monic polynomial \(f_m(x)\) of degree \(m\) and \(d_m\in R\) such that \(d_{m-1}\mid d_m\) and 
	\[
	\left\{1,\frac{f_1(\alpha)}{d_1},\ldots,\frac{f_m(\alpha)}{d_m}\right\}
	\] 
	is a basis for \(S_{m+1}\) over \(R\).
	
	Consider the map \(\pi\colon F_{m+1}\to M\) where  \(M\) is the \(R\)-module
	\[
	\left\{\left.
	r\frac{\alpha^m
	}{d}\right\vert r\in R
	\right\}
	\]
	given by projection on the \((m+1)^{\textrm{th}}\) component, that is
	\[
	\frac{a_0}{d}+a_1\frac{\alpha}{d}+\cdots+a_{m}\frac{\alpha^m}{d}\mapsto a_{m}\frac{\alpha^m}{d}
	\]
	Then, \(\pi\left(S_{m+1}\right)\) is a submodule of the free \(R\)-module \(M\) of rank one generated by \(\frac{\alpha^m}{d}\).  We have \(\pi\left(S_{m+1}\right)\neq 0\) since 
	\[
	\alpha^m=0\cdot \frac{1}{d}+0\cdot \frac{\alpha}{d}+\cdots+d\cdot \frac{\alpha^m}{d}\in S_{m+1}
	\] and \(\pi\left(\alpha^m\right)=d\,\frac{\alpha^m}{d}\neq 0\).  So, there is a \(\beta\in S_{m+1}\) such that \(\pi(\beta)\) generates \(\pi\left(S_{m+1}\right)\). 
	
	We will prove that 
	\[
	\left\{1,\frac{f_1(\alpha)}{d_1}, \cdots,\frac{f_{m-1}(\alpha)}{d_{m-1}},\beta\right\}
	\] is a basis for \(S_{m+1}\).
	
	Let \(\gamma\in S_{m+1}\).  Then
	\(\pi(\gamma)=s\pi(\beta)\). So, \(\pi(\gamma-s\beta)=0\) and  
	\(\gamma-s\beta\in S_m\).  So, by induction hypothesis 
	\[
	\gamma-s\beta=b_0+b_1\frac{f_1(\alpha)}{d_1}+\cdots+b_{m-1}\frac{f_{m-1}(\alpha)}{d_{m-1}}
	\]
	or
	\[
	\gamma=b_0+b_1\frac{f_1(\alpha)}{d_1}+\cdots+b_{m-1}\frac{f_{m-1}(\alpha)}{d_{m-1}}+s\beta
	\]Since \(\pi(\beta)\neq 0\), \(\beta\) is a polynomial in \(\alpha\) of degree \(m\) over \(K\). It follows that 
	\[
	\left\{1,\frac{f_1(\alpha)}{d_1}, \cdots,\frac{f_{m-1}(\alpha)}{d_{m-1}},\beta\right\}
	\] is a basis for \(S_{m+1}\). We have to prove that \(\beta\) is in the correct form.
	
	We have \(\frac{f_{m-1}(\alpha)}{d_{m-1}}\in S_m\) so, \(\frac{\alpha f_{m-1}(\alpha)}{d_{m-1}}\in S\). Also, since \(f_{m-1}(x)\) has degree \(m-1\), \(xf_{m-1}(x)\) has degree \(m\) and \(\frac{\alpha f_{m-1}(\alpha)}{d_{m-1}}\in S_{m+1}\).  So, 
	\[
	\frac{\alpha^m}{d_{m-1}}=\pi\left(\frac{\alpha f_{m-1}(\alpha)}{d_{m-1}}\right)\in \pi\left(S_{m+1}\right)
	\] So, \(\frac{\alpha^m}{d_{m-1}}=k\pi(\beta)\). Writing \(d_m=kd_{m-1}\), we have \(\pi(\beta)=\frac{\alpha^m}{d_m}\).   
	
	Since \(\beta\in S_{m+1}\), \(\beta=\sum_{i=0}^m\frac{u_i}{d}x^i\) with \(u_i\in R\).  From \(\pi(\beta)=\frac{\alpha^m}{d_m}\), it follows that \(u_m=\frac{d}{d_m}\) and \(\frac{d}{d_m}\in R\). Let
	\[
	f_m(x)=d_m\left(\frac{u_0}{d}+\frac{u_1}{d}x+\cdots+\frac{u_{m-1}}{d}x^{m-1}\right)+x^m
	\]
	Then \(\beta=\frac{f_m(\alpha)}{d_m}\). We need to show that \(f_{m}(x)\in R[x]\).    We have
	\(\frac{f_m(\alpha)}{d_{m-1}}=k\beta\in S\), so, 
	\[
	\frac{f_m(\alpha)-\alpha f_{m-1}(\alpha)}{d_{m-1}}=\gamma\in S
	\] Actually, \(\gamma\in S_m\). Let \[\gamma=u_0+\frac{u_1}{d_1}f_1(\alpha)+\cdots+\frac{u_{m-1}}{d_{m-1}}f_{m-1}(\alpha)=\frac{g(\alpha)}{d_{m-1}}\]
	where \(u_i\in R\) and 
	\[
	g(x)=d_{m-1}u_0+u_1\frac{d_{m-1}}{d_1}f_1(x)+u_2\frac{d_{m-1}}{d_2}f_2(x)+\cdots+u_{m-1}f_{m-1}(x)
	\]
	Since \(d_i\mid d_{m-1}\) if \(i\leq m-1\),  \(g(x)\in R[x]\) and \(\deg (g(x))<m\). We have \(g(\alpha)=f_{m}(\alpha)-\alpha f_{m-1}(\alpha)\).
	From \ref{key-lemma-1} it follows that \(g(x)=f_m(x)-xf_{m-1}(x)\). Since \(g(x)\), \(xf_{m-1}(x)\in R[x]\), it follows that \(f_m(x)\in R[x]\).
	
	We now show that \(d_m\) is uniquely determined up to a unit.  Let 
	\[
	I_m=\left\{x\in R\left\vert xS_{m+1}\subset R[\alpha]\right.\right\}
	\] We claim that \(I_m=\left(d_m\right)\).  We have \(d_m\in I_m\) because \(d_i\mid d_m\) for \(1\leq i\leq m-1\).  
	
	If \(\lambda\in I_m\), since \(\frac{\alpha^m}{d_m}\in S_{m+1}\), \(\lambda\frac{\alpha^m}{d_m}\in R[\alpha]\).  So, \(\frac{\lambda}{d_m}\alpha^m=h(\alpha)\), where \(h(x)\in R[x]\) and \(\deg(h(x))=m<n\).  So, \(h(x)=\frac{\lambda}{d_m}x^m\) and \(\frac{\lambda}{d_m}\in R\).
	This completes the proof of the main result in \ref{main-theorem}. 
	
	\textbf{Proof of \ref{basis3} of the \ref{main-theorem}}:  Since 
	\(\frac{g(\alpha)}{d_i}\in S\), we have 
	\[\frac{g(\alpha)}{d_i}=u_0+\sum_{i=1}^{n-1}u_i\frac{f_{i}\left(\alpha\right)}{d_i}\]Letting 
	\[
	h(x)=u_0+\sum_{i=1}^{n-1}u_i\frac{f_{i}\left(x\right)}{d_i}
	\]
	we have \(g(\alpha)=h(\alpha)\) and both \(g(x)\) and \(h(x)\) have degree less than \(n\).  Applying \ref{key-lemma-1}, it follows immediately that \(\frac{g(x)}{d_i}=h(x)\). Considering the coefficients of \(x^{n-1}\) both sides we conclude that \(u_{n-1}=0\).  Similarly, considering the coefficients of \(x^{n-2}\), \(\ldots\), \(x^{i+1}\) in succession, we get \(u_{n-2}=0\), \(u_{n-3}=0\), \(\ldots\), \(u_{i+1}=0\) and \(u_i=1\). So, the transition matrix from 
	\begin{equation}\label{ibasis1}
		\left\{1,\frac{f_1(\alpha)}{d_2}, \ldots, \frac{f_{i-1}(\alpha)}{d_{i-1}},\frac{g(\alpha)}{d_i},\frac{f_{i+1}(\alpha)}{d_{i+1}},\ldots, \frac{f_{n-1}(\alpha)}{d_{n-1}}\right\}
	\end{equation}  to 
	\begin{equation}\label{ibasis2}\left\{1,\frac{f_1(\alpha)}{d_1},\ldots,\frac{f_{n-1}(\alpha)}{d_{n-1}}\right\}
	\end{equation}
	has coefficients in \(R\) and is an upper triangular matrix with 1's along the  diagonal. Indeed all the columns of the matrix are the same as that of the same as the \(n\times n\) identity matrix except the i\textsuperscript{th} column which is 
	\[
	\begin{bmatrix}
		u_0\\
		u_1\\
		\vdots\\
		u_{i-1}\\
		1\\
		0\\
		\vdots\\
		0
	\end{bmatrix}
	\]So, the set in \ref{ibasis1} is also a basis for \(S\) over \(R\).
	
	\textbf{Proof of property \ref{basis1}}: Suppose \(\frac{g(\alpha)}{q}\in 
	S\) for a monic polynomial \(g(x)\in R[x]\) of degree \(i\), \(1\leq i\leq n-1\).  Writing \(\frac{g(\alpha)}q{}\) in terms of the integral basis in \ref{NIM}, we get 
	\[
	\frac{g(\alpha)}{q}=u_0+\sum_{i=1}^{n-1}u_i\frac{f_i(\alpha)}{d_i}, \ u_i\in R
	\]
	Arguing as we did in the proof of Property \ref{basis3}, we get 
	\(u_{n-1}\), \(u_{n-2}\), \(\ldots\), \(u_{i+1}\) are all zero and \(\frac{1}{q}=\frac{u_i}{d_i}\) or \(d_i=qu_i\) and \(q\mid d_i\).
	
	\textbf{Proof of Property \ref{basis1}}  We have \(\frac{f_i(\alpha)f_j(\alpha)}{d_id_j}\in S\) and \(g(x)=f_i(x)f_j(x)\) is a monic polynomial of degree \(i+j\) over \(R\). The result now follows from Property \ref{basis1}.
	
	\textbf{Proof of Property \ref{basis2}}:\ref{NIM} is an integral basis for \(S\)
	over \(R\) it follows that
	\[
	\disc(S/R)=\disc\left(\left\{1,\frac{f_1(\alpha)}{(d_1},\frac{f_2(\alpha)}{d_2},\ldots,\frac{f_{n-1}(\alpha)}{d_{n-1}}\right\}\right)
	\] The matrix over \(R\) that maps \(\left\{1,\alpha,\ldots,\alpha^{n-1}\right\}\) to the basis in \ref{NIM} is an upper triangular matrix with entries 1, \(\frac{1}{d_1}\), \(\frac{1}{d_2}\), \(\ldots\),\(\frac{1}{d_{n-1}}\) along the diagnal.  So, the determinant of the matrix is \(\frac{1}{d_1d_2\cdots d_{n-1}}\). Therefore,
	\[
	\disc(S/R)=\frac{1}{\left(d_1d_2\cdots d_{n-1}\right)^2}\disc(\alpha)
	\] which is the required result.
	
	\textbf{Proof of Property \ref{basis4}}: From Property \ref{basis1}, it follows by induction that \(d_1^i\mid d_i\) for \(1\leq i\leq n-1\).  So, \(d_1^{\frac{n(n-1)}{2}}\) divides \(d_1d_2\cdots d_{n-1}\).  From Property \ref{basis2} it follows that \(d_1^{n(n-1)}\) divides \(\disc(\alpha)\).
\end{proof}
\noindent{\large \bf{Acknowledgement}}:  The authors would like to thank 
Prof. R. Balasubramanian for many useful discussions.
This work was done when the second author was visiting 
The Institute of Mathematical Sciences, Madras. He thanks the
institute for its hospitality.


%
%

\bibliographystyle{sn-aps}      
\bibliography{qua4-rev-springer.bib}   

%
%

\end{document}